\documentclass[12pt,leqno]{amsart}
\usepackage{amsmath,amsthm,amsfonts,amssymb, amscd,amsbsy,multirow, hyperref}
\usepackage{hyperref}
\usepackage{graphicx}
\usepackage{xcolor}
\graphicspath{ }
\usepackage{wrapfig}
\usepackage{accents}
\usepackage{caption}
\usepackage{subcaption}
\usepackage{calligra}

\numberwithin{figure}{section}

\theoremstyle{plain}
\newtheorem{thm}{Theorem}[section]

\newtheorem{prop}[thm]{Proposition}
\newtheorem{cor}{Corollary}[thm]
\theoremstyle{definition}
\newtheorem{defn}{Definition}[section]

\theoremstyle{remark}

\usepackage{mathtools}
\title[Gradient $\rho$-Einstein solitons]{ Splitting theorem of Gradient $\rho$-Einstein solitons}
\author[A. A. Shaikh, P. Mandal and C. K. Mondal]{Absos Ali Shaikh$^{*1}$, Prosenjit Mandal$^{2}$ and Chandan Kumar Mondal $^{3}$}
\address{$^{1,2}$Department of Mathematics,\newline University of
Burdwan, Golapbag,\newline Burdwan-713104,\newline West Bengal, India.}

\address{$^3$Department of Mathematics,\newline School of Sciences,\newline Durgapur Regional Center\newline Netaji Subhas Open University,\newline Durgapur-713214,\newline West Bengal, India.}
\email{$^1$aask2003@yahoo.co.in, aashaikh@math.buruniv.ac.in}
\email{$^2$prosenjitmandal235@gmail.com}
\email{$^3$chan.alge@gmail.com, chandanmondal@wbnsou.ac.in}
\begin{document}
\begin{abstract}
In this paper, we have proved a weighted Laplacian comparison of distance function for manifolds with Bakry-\'Emery curvature bounded from below. Next, we have shown that a gradient $\rho$-Einstein soliton with a bounded integral condition on Ricci curvature splits off a line isometrically. Moreover, using this result, we have established some boundedness conditions on scalar curvature of gradient $\rho$-Einstein soliton. 
\end{abstract}
\noindent\footnotetext{${*}$ Corresponding author.\\
$\mathbf{2010}$\hspace{5pt}Mathematics\; Subject\; Classification: 53C21; 53C25; 53C44; 35K15.} 
\maketitle

\section{Introduction}
The Ricci flow, also known as Hamilton's Ricci flow in differential geometry, is a particular form of the partial differential equation with a Riemannian metric that is extensively explored by many geometers. It is sometimes referred to as an analogue to the heat equation because of the formal resemblance in the mathematical structure of the equation. Nonetheless, it exhibits a number of phenomena not encountered in the study of the heat equation. It was initially devised by Hamilton and is known as the Ricci flow after the appearance of the Ricci tensor in its explanation. He used it to establish a three-dimensional sphere theorem in the early 1900s \cite{HA82}. In 2002 and 2003, Perelman \cite{PE02}\cite{PE03} published various new findings on the Ricci flow, including a unique form of several technical features of Hamilton's approach that was not previously known. As a result of Hamilton and Perelman's work, the Thurston conjecture, which includes as a specific instance the Poincar\'e conjecture, which had been a well-known open problem in the subject of geometric topology since 1904, is now commonly recognized as having been proved.
For study on Ricci flow, see \cite{Chow-B}.
\par The splitting theorem of Riemannian manifolds was first proposed by Cheeger and Gromoll \cite{CG71}. They showed that if a Riemannian manifold has nonnegative Ricci curvature, it splits isometrically. Innami \cite{IN82} proved the splitting theorem for Riemannian manifold admitting a non-trivial affine function. Later, Mondal and Shaikh \cite{CA2020} generalized this by establishing in case of convex functions. Munteanu and Wang \cite{MW14} proved that a manifold with Bakry-\'Emery curvature bounded from below and the potential function with quadratical
growth  will split off a line isometrically. They also established an weighted Laplacian comparison
theorem for manifolds with Bakry-\'Emery curvature bounded from below. In this paper we have generalized the results of Munteanu and Wang \cite{MW14} for gradient $\rho$-Einstein solitons and derived a splitting theorem (see, Theorem \ref{th3}). For a discussion on splitting theorem in Ricci soliton, see the survey article \cite{LG20}.
\par The paper is structured as follows. First section deals with various preliminary notions and definitions. In section 2, estimation of distance function in $\rho$-Einstein soliton is obtained. Finally, in section 3, a splitting theorem of gradient $\rho$-Einstein soliton is provided. Moreover, an estimation of scalar curvature is also deduced for gradient $\rho$-Einstein soliton.
\section{Preliminaries and definitions}
The study of Ricci solitons is quite essential to explore the Ricci flow. An $n(\geq 2)$-dimensional Riemannian manifold $(M, g)$ with Riemannian metric $g$ is said to be a gradient Ricci soliton if it obeys the following condition:
\begin{equation}\label{rs1}
 Ric + \nabla^2 f = \mu g,
 \end{equation}
where $\nabla^2f$ stands for the Hessian of a smooth function $f\in C^{\infty}(M)$, $Ric$ is the Ricci curvature tensor and $\mu \in \mathbb{R}$. A Ricci soliton $(M, g)$ is called expanding if $\mu <0$, steady if $\mu=0$ or shrinking if $\mu > 0$. For some result of Ricci soliton see \cite{CA2019,CA2020,AC2021}. In general, it is natural to consider geometric flows of the following type on an $n(\geq 3)$- dimensional Riemannian manifold $(M,g)$:
$$\frac{\partial g}{\partial t}=-2({\rm Ric}-\rho Rg),$$
where $R$ denotes the scalar curvature of the metric $g$ and $\rho\in\mathbb{R}\diagdown\{0\}$. The parabolic theory for these flows was developed by Catino et al. \cite{Catino13}, which was first considered by Bourguignon \cite{Bour}. They called such a flow as Ricci-Bourguignon flows. They defined the following notion of
 $\rho$-Einstein soliton.
 
 \begin{defn}
Let $(M,g)$ be a Riemannian manifold of dimension $n(\geq 3)$, and let $\rho\in\mathbb{R}$, $\rho\neq 0$. Then $M$ is called a $\rho$-Einstein soliton if there is a smooth vector field $X$ such that

\begin{equation}\label{a2}
{\rm Ric}+\frac{1}{2}\mathcal{L}_Xg-\rho Rg=\mu g,
\end{equation}
where $Ric$ is the Ricci curvature tensor, $\mu$ is a constant and  $\mathcal{L}_Xg$ represents the Lie derivative of $g$ in the direction of the vector field $X$. Throughout the paper such
a $\rho$-Einstein soliton will be denoted by $(M, g, X)$.
\end{defn}
If there exists a smooth function $f:M\rightarrow\mathbb{R}$ such that $X=\nabla f$ then the $\rho$-Einstein soliton is called a gradient $\rho$-Einstein soliton, denoted by $(M,g,f)$ and in this case (\ref{a2}) takes the form 
 
\begin{equation}\label{res1}
 Ric+\nabla^2 f-\rho Rg=\mu g.
\end{equation}
The function $f$ is called a $\rho$-Einstein potential of the gradient $\rho$-Einstein soliton. As usual, a $\rho$-Einstein soliton is called steady, shrinking or expanding according as $\mu$ is zero, positive or negative respectively. By rescaling the metric $g$ we may assume that $\mu \in \{-\frac{1}{2}, 0,\frac{1}{2}\}$. For more information on $\rho$-Einstein solitons, see \cite{CA2020, Absos, AAP2021} and the references therein.
In particular, if $\rho=\frac{1}{2}$  (resp., $\frac{1}{n} \text{ or } \frac{1}{2(n-1)} )$, then  a gradient $\rho$-Einstein soliton  is called a gradient Einstein soliton (resp., a gradient traceless Ricci soliton or a gradient Schouten soliton).
\section{Estimation of Laplacian for the distance function}
The Bakry-\'Emery curvature is given by $$Ric_f=Ric+\nabla^2 f.$$ For a Riemannian manifold with a smooth measure, the Bakry-\'Emery tensor is an analogue of the Ricci tensor. This tensor and its relevance to diffusion processes were extensively examined by Bakry and \'Emery \cite{BE85}. The Bakry-\'Emery tensor may also be found in a variety of subjects, see e.g. \cite{MI68}. The weighted Laplacian is given by $$\Delta_\phi f=\Delta f-\langle \nabla \phi,\nabla f \rangle,$$
where $\phi$ is some smooth function on $M$.

Suppose $(M,g,e^{-f} dvol)$ as a smooth metric measure space such that $Ric+\nabla^2 f\geq\mu g+\rho Rg$, where $e^{-f} dvol$ is the Riemannian volume density on $M$.  In the geodesic coordinates, for a fixed point $p\in M$ and $r>0$, we denote the volume form by $$dV|_{exp_p(r\xi)}=J(p,r,\xi)dr d\xi,$$
where $\xi\in S_pM$, the unit tangent sphere at $p$. If $x\in M$ is any point outside the cut locus of $p$, such that $x=exp(r\xi)$, then
$$\Delta d(p,x)=\frac{d}{dr}ln J(p,r,\xi).$$
In the following, we will omit the dependence of these quantities on $p$ and $\xi$. If $\gamma$ be a minimizing geodesic starting from $p$ and $m(r):=\frac{d}{dr}lnJ(r)$, then along $\gamma$, we get (see e.g. \cite{Li93})
\begin{equation}\label{eq1}
m'(r)+\frac{1}{n-1}m^2(r)+Ric\big(\frac{\partial}{\partial r},\frac{\partial}{\partial r}\big)\leq 0,  
\end{equation}
here the differentiation is taken with respect to the variable $r$.
Multiplying  $(\ref{eq1})$ by $r^2$ and then integrating from $r=0$ to $r=t>0$ (see e.g.  \cite{GW09}), we obtain
\begin{equation}
\int_{0}^{t}m'(r)r^2 dr+\frac{1}{n-1}\int_{0}^{t}r^2m^2(r) dr+\mu\int_{0}^{t}r^2 dr+\rho\int_{0}^{t}Rr^2 dr\leq \int_{0}^{t}f''(r)r^2 dr,
\end{equation}
in the above relation, we have used $Ric+\nabla^2 f\geq\mu g+\rho Rg$ and $f''(r):=\nabla^2 f(\frac{\partial}{\partial r},\frac{\partial}{\partial r})=\frac{d^2}{dr^2}(f\circ\gamma)(r).$
After rearranging the terms, the above inequality takes the form
\begin{equation}
m(t)t^2+\frac{1}{n-1}\int_{0}^{t}\{m(r)r-(n-1)\}^2dr+\rho\int_{0}^{t}Rr^2\leq(n-1)t+t^2f'(t)-2\int_{0}^{t}f'(r)r dr-\mu\frac{t^3}{3}.
\end{equation}
Cancelling positive terms from left side of the above inequality and then dividing by $t^2$, it yields
\begin{equation}\label{eq2}
m(t)\leq\frac{n-1}{t}+f'(t)-\frac{2}{t^2}\int_{0}^{t}f'(s)s ds-\mu\frac{t}{3}-\frac{\rho}{t^2}\int_{0}^{t}s^2R ds.
\end{equation}
Integrating $\ref{eq2}$, from $t=\epsilon$ to $t=r>\epsilon$ for $\epsilon>0$ small, then it implies
\begin{eqnarray}
\nonumber ln J(r)-ln J(\epsilon)&\leq &(n-1)(ln r-ln \epsilon)-\frac{\mu}{6}(r^2-\epsilon^2)+f(r)-f(\epsilon)\\
&&-\int_{\epsilon}^{r}\frac{2}{t^2}\Big(\int_{0}^{t}sf'(s)ds\Big)dt-\int_{\epsilon}^{r}\frac{\rho}{t^2}\Big(\int_{0}^{t}s^2R ds\Big)dt.
\end{eqnarray}
Taking limit as $\epsilon\rightarrow 0$, above inequality yields
\begin{equation}
ln J(r)\leq(n-1)ln r-\frac{\mu}{6}r^2-\big(f(r)-f(0)\big)+\frac{2}{r}\int_{0}^{r}sf'(s)ds+\frac{\rho}{r}\int_{0}^{r}s^2 R ds-\rho\int_{0}^{r}s Rds.
\end{equation}
Again, rearranging the terms, we get
\begin{equation}
-\frac{2}{r^2}\int_{0}^{r}sf'(s)ds-\frac{\rho}{r^2}\int_{0}^{r}s^2 R ds\leq -\frac{1}{r}ln \Big( \frac{J(r)}{r^{n-1}}\Big)-\frac{\mu}{6}r-\frac{1}{r}\big(f(r)-f(0)\big)-\frac{\rho}{r}\int_{0}^{r}s Rds,
\end{equation}
with the help of $(\ref{eq2})$, the above inequality can be written as
\begin{equation}
m(r)\leq\frac{n-1}{r}+f'(r)-\mu\frac{r}{3}-\frac{1}{r}ln \Big( \frac{J(r)}{r^{n-1}}\Big)-\frac{\mu}{6}r-\frac{1}{r}\big(f(r)-f(0)\big)-\frac{\rho}{r}\int_{0}^{r}s Rds.
\end{equation}
The importance of the above equation lies in the fact that it now involves $f$ only at the two end points $p$ and $x$ of the geodesic $\gamma$.
Thus we can rewrite the above inequality in the following form:
\begin{equation}
\Delta_fd(p,x)\leq\frac{n-1}{r}-\mu\frac{r}{2}-\frac{1}{r}ln \Big( \frac{J(p,r,\xi)}{r^{n-1}}\Big)-\frac{1}{r}\big(f(x)-f(p)\big)-\frac{\rho}{r}\int_{0}^{r}s Rds,
\end{equation}
Thus we can write the following result from above relation:
\begin{thm}\label{lem1}
Let $ \gamma$ be a minimizing normal geodesic in $M$ with $p=\gamma (0)$ and $x=\gamma (r)$. If $Ric+\nabla^2 f\geq\mu g+\rho Rg$, then
\begin{equation*}
\Delta d(p,x)\leq\frac{n-1}{r}-\mu\frac{r}{2}+f'(r)-\frac{1}{r}ln \Big( \frac{J(r)}{r^{n-1}}\Big)-\frac{1}{r}\big(f(r)-f(0)\big)-\frac{\rho}{r}\int_{0}^{r}s Rds
\end{equation*}
and
\begin{equation*}
\Delta_fd(p,x)\leq\frac{n-1}{r}-\mu\frac{r}{2}-\frac{1}{r}ln \Big( \frac{J(p,r,\xi)}{r^{n-1}}\Big)-\frac{1}{r}\big(f(x)-f(p)\big)-\frac{\rho}{r}\int_{0}^{r}s Rds.
\end{equation*}
\end{thm}
\section{Splitting theorem of gradient $\rho$-Einstein solitons}
\begin{thm}\label{th3}
Let $(M,g,f)$ be a gradient $\rho$-Einstein soliton with $\rho R\geq 0$. If there is a geodesic line $\gamma:(-\infty,\infty)\rightarrow M$ satisfying
\begin{equation}\label{eq10}
\liminf\limits_{t\rightarrow \infty}\int_{0}^{t}Ric\big(\gamma'(s),\gamma'(s)\big) ds+\liminf\limits_{t\rightarrow -\infty}\int_{t}^{0}Ric\big(\gamma'(s),\gamma'(s)\big) ds\geq 0,
\end{equation}
and
\begin{equation*}
\int_{0}^{t}\int_{0}^{s}Rd\tau ds-\int_{0}^{r}s Rds \text{  is finite}
\end{equation*}
then $M$ splits off a line isometrically.
\end{thm}
\begin{proof}
Fix a compact domain $D$ in $M$, and denote $f(t)=f(\gamma(t))$ for simplicity. Applying the soliton equation $Ric+\nabla^2 f-\rho Rg=\mu g$ to $\gamma'$ and integrating twice, we have
\begin{equation}\label{eq3}
f(t)-f(0)=f'(0)t+\frac{\mu}{2}t^2-\int_{0}^{t}\int_{0}^{s}Ric\big(\gamma'(\tau),\gamma'(\tau)\big)d\tau ds+\rho\int_{0}^{t}\int_{0}^{s}Rd\tau
 ds.
\end{equation}
For $t>0$ and $x\in D$, $\tau_t(s)$ denotes the minimizing normal geodesic from $\gamma(t)$ to $x$. Let $r:=d(\gamma(t),x)$, then $\tau_t(0)=\gamma(t)$ and $\tau_t(r)=x$.
Using the Theorem \ref{lem1}, the following can be deduced:
\begin{equation}\label{eq6}
\Delta_fd(\gamma(t),x)+\frac{1}{r}ln J(\gamma(t),r,\tau'_t(0))\leq\frac{n-1}{r}(1+lnr)-\mu\frac{r}{2}+\frac{f(t)}{r}-\frac{f(x)}{r}-\frac{\rho}{r}\int_{0}^{r}s Rds.
\end{equation}
As in \cite{MW14}, for any $\epsilon>0$ and $t>0$ sufficiently large, we claim that 
\begin{eqnarray}
\frac{n-1}{r}(1+lnr)-\frac{\mu}{2}r+\frac{f(t)}{r}-\frac{f(x)}{r}&\leq& \mu(t-r)+\epsilon+f'(0)\frac{t}{r}+\frac{\rho}{r}\int_{0}^{t}\int_{0}^{s}Rd\tau ds\\
\nonumber&-&\frac{1}{r}\int_{0}^{t}\int_{0}^{s}Ric\big(\gamma'(\tau),\gamma'(\tau)\big)d\tau ds.
\end{eqnarray}
Since $x$ belongs to a fixed bounded domain $D$, for $t>0$ sufficiently large, the following holds:
\begin{equation}\label{eq4}
\frac{n-1}{r}(1+ln r)-\frac{f(x)}{r}\leq\frac{\epsilon}{2}.
\end{equation} 
In view of (\ref{eq3}), we obtain
\begin{eqnarray}
\nonumber&-&\frac{1}{2}\mu r+\frac{f(t)}{r}-\mu(t-r)=\frac{1}{4r}\big(-2\mu r^2+4f(t)-4\mu r(t-r)\big)\\
\nonumber&=&\frac{1}{4r}\Big(4f(0)+4f'(0)t+2\mu t^2-4\int_{0}^{t}\int_{0}^{s}Ric\big(\gamma'(\tau),\gamma'(\tau)\big)d\tau ds-2\mu r^2\\
\nonumber &&-4\mu r(t-r)+4\rho\int_{0}^{t}\int_{0}^{s}Rd\tau
 ds\Big)\\
 \nonumber&=&\frac{1}{4r}\Big(4f(0)+4f'(0)t+2\mu (t-r)^2-4\int_{0}^{t}\int_{0}^{s}Ric\big(\gamma'(\tau),\gamma'(\tau)\big)d\tau ds+4\rho\int_{0}^{t}\int_{0}^{s}Rd\tau
  ds\Big)\\
  \nonumber&\leq&\frac{1}{4r}\Big(4f(0)+4f'(0)t+d(x,\gamma(0))^2-4\int_{0}^{t}\int_{0}^{s}Ric\big(\gamma'(\tau),\gamma'(\tau)\big)d\tau ds+4\rho\int_{0}^{t}\int_{0}^{s}Rd\tau
    ds\Big).
\end{eqnarray}
So for $t$ sufficiently large, we have
\begin{equation}\label{eq5}
-\frac{1}{2}\mu r+\frac{f(t)}{r}-\mu(t-r)\leq\frac{\epsilon}{2}+f'(0)\frac{t}{r}-\frac{1}{r}\int_{0}^{t}\int_{0}^{s}Ric\big(\gamma'(\tau),\gamma'(\tau)\big)d\tau ds+\frac{\rho}{r}\int_{0}^{t}\int_{0}^{s}Rd\tau ds.
\end{equation}
Combining (\ref{eq4}) and (\ref{eq5}), we arrived at our claim.
Hence (\ref{eq6}) yields
\begin{eqnarray*}
\Delta_fd(x,\gamma(t))+\frac{1}{r}ln J(\gamma(t),r,\xi)&\leq& \mu(t-d(x,\gamma(t)))+\epsilon+f'(0)\frac{t}{r}+\frac{\rho}{r}\int_{0}^{t}\int_{0}^{s}Rd\tau ds\\
\nonumber&-&\frac{\rho}{r}\int_{0}^{r}s Rds-\frac{1}{r}\int_{0}^{t}\int_{0}^{s}Ric\big(\gamma'(\tau),\gamma'(\tau)\big)d\tau ds.
\end{eqnarray*}
After rearranging the terms, the above equation takes the form
\begin{eqnarray*}
\Delta_f\big(t-d(x,\gamma(t))\big)+\mu(t-d(x,\gamma(t)))&\geq&\frac{1}{r}ln J(\gamma(t),r,\xi) -\epsilon-f'(0)\frac{t}{r}-\frac{\rho}{r}\int_{0}^{t}\int_{0}^{s}Rd\tau ds\\
\nonumber&+&\frac{\rho}{r}\int_{0}^{r}s Rds+\frac{1}{r}\int_{0}^{t}\int_{0}^{s}Ric\big(\gamma'(\tau),\gamma'(\tau)\big)d\tau ds.
\end{eqnarray*}
Using triangle inequality, it is easy to see that there is a $c>0$ such that $D\subseteq B(\gamma(t),t+c)\setminus B(\gamma(t),t-c)$, so $\frac{t}{r}\rightarrow 1$ uniformly when $x\in D$.
Choosing a non-negative smooth finction with compact support $\phi\in C^{\infty}_0(D)$ and multiplying $\phi J(\gamma(t),r,\xi)$ on both sides of the above inequality, we obtain
\begin{eqnarray*}
&&\int_{M}\Delta_f\big(t-d(x,\gamma(t))\big)\phi dvol+\int_{M}\mu(t-d(x,\gamma(t)))\phi dvol\\
&&\geq\int_{0}^{\infty}\int_{S^{n-1}}\frac{ln J(\gamma(t),r,\xi)}{r}\phi J(\gamma(t),r,\xi) d\xi dr -\epsilon   \int_{M}\phi dvol-f'(0)\int_{M}\frac{t}{r}\phi dvol\\
&&-\int_{0}^{\infty}\int_{S^{n-1}}\frac{\rho \int_{0}^{t}\int_{0}^{s}Rd\tau ds}{r}\phi J(\gamma(t),r,\xi) d\xi dr+\int_{0}^{\infty}\int_{S^{n-1}}\frac{\rho \int_{0}^{r}s Rds}{r}\phi J(\gamma(t),r,\xi) d\xi dr\\
&&+\int_{0}^{\infty}\int_{S^{n-1}}\frac{\int_{0}^{t}\int_{0}^{s}Ric\big(\gamma'(\tau),\gamma'(\tau)\big)d\tau ds}{r}\phi J(\gamma(t),r,\xi) d\xi dr\\
&&\geq -\frac{1}{e}\int_{0}^{\infty}\int_{S^{n-1}}\frac{\phi}{r} d\xi dr -\epsilon   \int_{M}\phi dvol-f'(0)\int_{M}\frac{t}{r}\phi dvol\\
&&-\int_{0}^{\infty}\int_{S^{n-1}}\Big\{\frac{\rho \int_{0}^{t}\int_{0}^{s}Rd\tau ds-\rho \int_{0}^{r}s Rds}{r}\Big\}\phi J(\gamma(t),r,\xi) d\xi dr\\
&&+\int_{0}^{\infty}\int_{S^{n-1}}\frac{\int_{0}^{t}\int_{0}^{s}Ric\big(\gamma'(\tau),\gamma'(\tau)\big)d\tau ds}{r}\phi J(\gamma(t),r,\xi) d\xi dr\\
&&\geq -\frac{1}{e\cdot t}\int_{0}^{\infty}\int_{S^{n-1}}\frac{\phi \cdot t}{r} d\xi dr -\epsilon   \int_{M}\phi dvol-f'(0)\int_{M}\frac{t}{r}\phi dvol\\
&&-\frac{\rho \int_{0}^{t}\int_{0}^{s}Rd\tau ds-\rho \int_{0}^{r}s Rds}{t}\int_{0}^{\infty}\int_{S^{n-1}}\frac{\phi \cdot t}{r} J(\gamma(t),r,\xi) d\xi dr\\
&&+\frac{\int_{0}^{t}\int_{0}^{s}Ric\big(\gamma'(\tau),\gamma'(\tau)\big)d\tau ds}{t}\int_{0}^{\infty}\int_{S^{n-1}}\frac{t}{r}\phi J(\gamma(t),r,\xi) d\xi dr,
\end{eqnarray*}
in the second inequality we have used ($y$ $ln y)\geq -\frac{1}{e}$ for $y> 0$.
The Busemann function associated to $\gamma:[0,\infty)\rightarrow M$ \big(resp., $\gamma:(-\infty,0]\rightarrow M$ \big) is given by
$$\beta_+(x)=\lim\limits_{t\rightarrow \infty}(t-d(x,\gamma (t))) \big(\text{ resp., } \beta_-(x)=\lim\limits_{t\rightarrow \infty}(t-d(x,\gamma (-t)))\big).$$
Since
\begin{equation*}
\liminf\limits_{t\rightarrow \infty}\frac{\int_{0}^{t}\int_{0}^{s}Ric\big(\gamma'(\tau),\gamma'(\tau)\big)d\tau ds}{t}\geq\liminf\limits_{t\rightarrow \infty}\int_{0}^{t}Ric\big(\gamma'(s),\gamma'(s)\big) ds,
\end{equation*}
 considering $\liminf_{t\rightarrow \infty}$ on the above inequality, it follows that the following inequality hods in the sense of distribution
\begin{equation*}
\Delta_f\beta_++\mu\beta_+\geq-\epsilon-f'(0)+\liminf\limits_{t\rightarrow \infty}\int_{0}^{t}Ric\big(\gamma'(s),\gamma'(s)\big) ds.
\end{equation*}
Since $\epsilon$ is arbitrary, hence
\begin{equation}\label{eq7}
\Delta_f\beta_++\mu\beta_+\geq-f'(0)+\liminf\limits_{t\rightarrow \infty}\int_{0}^{t}Ric\big(\gamma'(s),\gamma'(s)\big) ds.
\end{equation}
Similarly, we get
\begin{equation}\label{eq8}
\Delta_f\beta_-+\mu\beta_-\geq f'(0)+\liminf\limits_{t\rightarrow -\infty}\int_{t}^{0}Ric\big(\gamma'(s),\gamma'(s)\big) ds.
\end{equation}
Define $\beta=\beta_++\beta_-$, and then adding $(\ref{eq7})$ and $(\ref{eq8})$, we obtain
\begin{equation}\label{eq9}
\Delta_f\beta+\mu\beta\geq\liminf\limits_{t\rightarrow \infty}\int_{0}^{t}Ric\big(\gamma'(s),\gamma'(s)\big) ds+\liminf\limits_{t\rightarrow -\infty}\int_{t}^{0}Ric\big(\gamma'(s),\gamma'(s)\big) ds.
\end{equation}
Now from the definition of the Busemann function and $(\ref{eq10})$, we have $\beta\leq 0$ and $\Delta_f\beta+\mu\beta\geq 0.$
When $\mu=\frac{1}{2}$, integrating with respect to $e^{-f}dvol$ we get $\beta=0$. For steady or expanding case, applying the strong maximum principle, we have $\beta=0$.
Hence for a gradient $\rho$-Einstein soliton, $\beta_+=-\beta_-$.\\
Thus by virtue of $(\ref{eq10})$ and $(\ref{eq8})$, we have
\begin{equation*}
\Delta_f(-\beta_+)+\mu(-\beta_+)\geq f'(0)-\liminf\limits_{t\rightarrow \infty}\int_{0}^{t}Ric\big(\gamma'(s),\gamma'(s)\big) ds,
\end{equation*}
i.e.,
\begin{equation*}
\Delta_f\beta_++\mu\beta_+\leq-f'(0)+\liminf\limits_{t\rightarrow \infty}\int_{0}^{t}Ric\big(\gamma'(s),\gamma'(s)\big) ds.
\end{equation*}
Hence, combining the above inequality and $(\ref{eq7})$, it follows that
\begin{equation}
\Delta_f\beta_++\mu\beta_+=-f'(0)+\liminf\limits_{t\rightarrow \infty}\int_{0}^{t}Ric\big(\gamma'(s),\gamma'(s)\big) ds.
\end{equation}
Now using Bochner formula for weighted Laplacian, we get
\begin{eqnarray}
\frac{1}{2}\Delta_f|\nabla \beta_+|^2&&=|\nabla
^2 \beta_+|^2+\langle\nabla\Delta_f \beta_+,\nabla\beta_+\rangle+Ric_f(\nabla\beta_+,\nabla\beta_+)\\
&&=|\nabla^2 \beta_+|^2-\mu|\nabla \beta_+|^2+\mu|\nabla \beta_+|^2+\rho R|\nabla \beta_+|^2\\
&&=|\nabla^2 \beta_+|^2+\rho R|\nabla \beta_+|^2.
\end{eqnarray}
As $|\nabla\beta_+|=1$, we get $|\nabla^2\beta_+|^2=-\rho R$. Since $\rho R \geq 0$, hence $\nabla^2 \beta_+=0$. This completes the proof.
\end{proof}
The following corollary follows from the above theorem:
\begin{cor}\label{51}
Let $(M,g,f)$ be a gradient $\rho$-Einstein soliton such that there is a geodesic line $\gamma:(-\infty,\infty)\rightarrow M$ satisfying
\begin{equation*}
\liminf\limits_{t\rightarrow \infty}\int_{0}^{t}Ric\big(\gamma'(s),\gamma'(s)\big) ds+\liminf\limits_{t\rightarrow -\infty}\int_{t}^{0}Ric\big(\gamma'(s),\gamma'(s)\big) ds\geq 0
\end{equation*}
and
\begin{equation*}
\int_{0}^{t}\int_{0}^{s}Rd\tau ds-\int_{0}^{r}s Rds \text{  is finite}.
\end{equation*}
Then the scalar curvature $R\leq 0$ $(\text{ resp., } \geq 0)$ for $\rho>0$ $(\text{ resp., } <0)$.
\end{cor}
\begin{prop}\cite{VB2021}\label{sprop1}
Let $(M, g, f)$ be a complete non-compact non-steady Schouten soliton with a non-constant potential function $f$. Then
$$0\leq \lambda R\leq 2(n-1)\lambda^2.$$
\end{prop}
Now from Theorem \ref{th3} and Proposition \ref{sprop1} we can state the following:
\begin{cor}
Let $(M, g, f)$ be a complete non-compact non-steady Schouten soliton with a non-constant potential function $f$ such that there is a geodesic line $\gamma:(-\infty,\infty)\rightarrow M$ satisfying
\begin{equation*}
\liminf\limits_{t\rightarrow \infty}\int_{0}^{t}Ric\big(\gamma'(s),\gamma'(s)\big) ds+\liminf\limits_{t\rightarrow -\infty}\int_{t}^{0}Ric\big(\gamma'(s),\gamma'(s)\big) ds\geq 0
\end{equation*}
and
\begin{equation*}
\int_{0}^{t}\int_{0}^{s}Rd\tau ds-\int_{0}^{r}s Rds \text{  is finite}.
\end{equation*}
Then $M$ is expanding and hence $2(n-1)\mu\leq R\leq 0.$
\end{cor}
\section{acknowledgment}
 The second author gratefully acknowledges to the CSIR(File No.:09/025(0282)/2019-EMR-I), Govt. of India for the award of JRF. Also the third author conveys sincere thanks to the Netaji Subhas Open
 University for partial financial assistance (Project No.: AC/140/2021-22).

\end{document}